\documentclass[12pt,a4paper]{amsart}
\usepackage{amscd,amssymb,amsmath,latexsym}
\usepackage{graphicx,epsfig,color,url}
\usepackage[all]{xy} 
\usepackage{todonotes}
\usepackage{tikz}
\usepackage{hyperref}%
\usepackage{slashbox}
\usepackage{adjustbox}
\usepackage{graphicx}
\usepackage{enumerate}
\hypersetup{
  colorlinks=true,
  breaklinks=true,
  urlcolor= blue,
  linkcolor= blue,
}

\newcommand{\codim}{\operatorname{codim}}

\newcommand{\Spec}{\operatorname{Spec}}
\newcommand{\ord}{\operatorname{ord}}

\newcommand{\Sym}{\operatorname{Sym}}

\newcommand{\A}{\mathbb{A}}

\newcommand{\G}{\mathbb{G}}
\newcommand{\K}{\mathbb{K}}

\newcommand{\N}{\mathbb{N}}
\renewcommand{\P}{\mathbb{P}}

\newcommand{\Z}{\mathbb{Z}}

\newcommand{\cE}{{\mathcal E}}

\newcommand{\cI}{{\mathcal I}}

\newcommand{\cO}{{\mathcal O}}
\newcommand{\cP}{{\mathcal P}}

\newcommand{\mf}{{\mathfrak{m}}}

\numberwithin{equation}{section}

\theoremstyle{definition}
\newtheorem{definition}{Definition}[section]

\newtheorem{remark}[definition]{Remark}
\newtheorem{example}[definition]{Example}

\theoremstyle{plain}
\newtheorem{lemma}[definition]{Lemma}
\newtheorem{proposition}[definition]{Proposition}
\newtheorem{theorem}[definition]{Theorem}
\newtheorem{corollary}[definition]{Corollary}

\makeatletter
\providecommand\@dotsep{5}
\makeatother

\begin{document}

\title{Hyperflex loci of hypersurfaces}

\author{Cristina Bertone}
\address{Dipartimento di Matematica, Universit\`a di Torino, Torino, Italy}
\email{cristina.bertone@unito.it}
\urladdr{https://sites.google.com/view/cristinabertone/home} 

\author{Martin Weimann}
\address{LMNO, University of Caen Normandie, Caen, France
}
\email{martin.weimann@unicaen.fr}
\urladdr{https://weimann.users.lmno.cnrs.fr/} 

\begin{abstract}
The $k$-flex locus of a projective hypersurface $V\subset \P^n$ is the locus of points $p\in V$ such that  there is a line with order of contact at least $k$ with  $V$ at   $p$.
Unexpected contact orders occur when $k\ge n+1$. The case $k=n+1$ is known as the classical flex locus, which has been studied in details in the literature. This paper is dedicated to compute the dimension and the degree of the $k$-flex locus of a general degree $d$ hypersurface for any value of $k$. As a corollary, we compute the dimension and the degree of the biggest ruled subvariety of a general hypersurface. 
We  show moreover  that through a generic $k$-flex point passes a unique $k$-flex line and that this line has contact order exactly $k$ if $k\le d$. The proof is based on the computation of the top Chern class of a certain vector bundle of relative principal parts, inspired by and generalizing a work of Eisenbud and Harris. 
\end{abstract}


\maketitle


\section{Introduction}
Let $V\subset \P^n$ be a reduced projective hypersurface defined over an algebraically closed field $\K$ of caracteristic zero and let $k\in \N$. 
 The $k$-flex locus of $V$ 
 is the locus of points $p\in V$ 
 such that  there is a line with order of contact at least $k$ with  $V$ at  $p$
 (see Definition \ref{def:1} for a precise statement). It is not difficult to see that this locus is a closed subvariety $V_k\subseteq V$. 
 
It is well known \cite{BDSW} that $V_k=V$ when $k\le n$ and the first interesting case occurs when $k = n+1$, in which case the subvariety $V_{n+1}\subset V$ is classically called the \emph{flex locus} of $V$. For instance, the flex locus of a projective plane curve coincides with the usual subset of inflexion points, and is determined by the vanishing of the hessian determinant of the homogeneous polynomial defining the curve. 
The study of the flex locus of curves and surfaces is a classical subject of algebraic geometry, studied from the XIXth century by Monge, Salmon and Cayley among others. In the last decades, there has been a regain of interest to this subject due to its applications in incidence geometry  \cite{BDSW, EisHar,GK15,Kat14,Kol15, SS18,Tao14}. In the recent article \cite{BDSW}, the authors show that the flex locus of a general hypersurface $V\subset \P^n$ of degree $d\ge n$ is a codimension one subvariety of $V$ and they obtain an explicit formula for its degree in terms of the parameters $d$ and $n$, generalizing the famous Salmon's degree formula $11d^2-24d$ for the case of surfaces in $\P^3$ \cite{salmon}.

In this paper, we generalize the results of \cite{BDSW} to any $k$, dealing with \emph{hyperflex loci}\footnote{In \cite{EisHar}, the terminology \emph{hyperflex} is rather used for values of $k$ such that $V_k$ is expected to be empty.}, high values of $k$ corresponding to a \lq\lq highly non expected\rq\rq contact order. 
Let us remark that if $V$ has degree $d$, a line with contact order at least $d+1$ with $V$ is necessarily contained in $V$ by Bézout's theorem. Hence, any degree $d$ hypersurface admits a filtration
\begin{equation}\label{eq:stratification}
V=V_1=\cdots =V_n \supseteq V_{n+1}  \supseteq \,\, \cdots \,\,\supseteq V_{d}\supseteq V_{d+1} =V_{d+2}=\cdots = V_{\infty},
\end{equation}
where $V_{\infty}$ is the biggest ruled (union of lines) subvariety  of $V$. Thus $k=d+1$ is the highest interesting value to consider. Note that if $d<n$, the hypersurface is necessarily ruled, and there is nothing to do.

\subsection*{Main results.} This paper is dedicated to compute the dimension and  the degree of each $V_k$ for a \emph{general degree $d$ hypersurface} $V$, that is for all $V$ belonging to a certain non empty Zariski open set of the space $\P^N$ of all degree $d$ hypersurfaces of $\P^n$. We prove :

\begin{theorem}\label{thm:main} Let $V\subset \P^n$ be a general hypersurface of degree $d$. Let $k\in \N$ such that $n+1 \le k \le d+1$, which by \eqref{eq:stratification} are the only meaningful values.

\begin{enumerate}[(i)]
\item\label{it1thmMain} If $k>2n-1$, the $k$-flex locus is empty. Otherwise, $V_k\subset V$ is a subvariety of pure dimension $$\dim V_k=2n-k-1,$$
except if $k=2n-1=d+1$ in which case $V_k$ is also empty.
\item\label{it2thmMain} Assuming $V_k$ non empty, its degree is
 \begin{equation}\label{eq:main_theorem}
\deg(V_k)=N_k(n,d):=\sum_{m=n-1}^{k-1} \lambda_{m}(d,k)\left(\binom{m}{n-1}-\binom{m}{n}\right), 
\end{equation}
where the coefficients  $\lambda_m(d,k)\in \N$ are defined by the polynomial equality $$d\prod_{j=1}^{k-1}\left(jX+(d-2j)\right)=\sum_{m=0}^{k-1}\lambda_m(d,k)X^m.$$
\item\label{it3thmMain} Through a general $p\in V_k$, there is a unique $k$-flex line (Definition \ref{def:1}). This line has contact order exactly $k$ with $V$ at $p$ if $k\le d$, and is contained in $V$ otherwise.
\end{enumerate}
\end{theorem}

\noindent
For $m=n-1$, the second binomial coefficient in \eqref{eq:main_theorem} is $\binom{n-1}{n}=0$ by convention. 

\begin{corollary}\label{cor:1}
The biggest ruled subvariety $V_{\infty}$ of a general degree $d$ hypersurface $V$ is empty if $d\ge 2n-2$. Otherwise, 
$$\dim V_\infty=2n-d-2\quad {\rm and}\quad \deg V_\infty=N_{d+1}(n,d).
$$
In particular, a general hypersurface of degree $d=2n-3$ contains exactly $N_{2n-2}(n,2n-3)$ lines.
\end{corollary}

%
%

When $k=n+1$, we recover the degree formula of the classical flex locus of $V$ obtained in \cite{BDSW}, see Corollary \ref{cor:deg_flex_locus}. The proof in \cite{BDSW} is based on the theory of multidimensional resultants which appears to be useful since the flex locus is a codimension one subvariety of $V$. However, let us emphasize that the flex case is already subtle since resultants only lead to affine equations of the flex locus (Remark \ref{rem:flex_vs_hyperflex}). For higher values of $k$, the $k$-flex locus has higher codimension  and using resultants seem hopeless. We rather follow a strategy inspired by Eisenbud and Harris in \cite[Chapter 11]{EisHar} : we compute the degree of the $k$-flex locus in terms of the top Chern class of a suitable \emph{vector bundle of relative principal parts} $\mathcal{E}\to \Phi$ over the incidence subvariety $\Phi\subset \P^n\times \G(1,n)$.

\subsection*{Organisation of the paper}
Section \ref{sec2} is dedicated to give explicit local equations of the incidence variety $\Gamma_k\subset \P^N\times \P^n\times  \G(1,n)$ given by those triplets $(V,p,L)$ such that the line $L$ has contact order at least $k$ with $V$ at $p$. This result is used in Section \ref{sec3}, where we show that the $k$-flex locus of a general hypersurface has the expected dimension, and that the $k$-flex line at a general $k$-flex point is unique and has the expected contact order (Theorem \ref{thm:generic}). In Section \ref{sec4}, we introduce the vector bundle of relative principal parts  $\mathcal{E}\to \Phi$ and we explain its relation with our problem :  any degree $d$ hypersurface $V=\{f=0\}$ determines a global section $\tau_f\in H^0(\Phi,\mathcal{E})$  whose value at $(p,L)$ is the restriction of $f$ to the $k^{th}$-infinitesimal neighborhood of $L$ at $p$ (Theorem \ref{thm:principal_part}). This key result is mainly \cite[Theorem 11.2]{EisHar}, but we sketch a self-contained proof. We deduce that the degree of the top Chern class of $\mathcal{E}$ determines the degree of the $k$-flex locus of a general hypersurface of fixed degree. In Section \ref{sec5}, after reminding the main properties of Chern classes (Theorem \ref{thmChernClasses}), we follow \cite{EisHar} to compute the top Chern class of $\mathcal{E}$ (Theorem \ref{thm:degree_in_terms _of_Chern_classes}) using Whitney's formula together with Schubert calculus to perform  computations in the Chow ring of the Grassmannian. We give a closed formula for some products of Schubert classes in terms of entries of Catalan's trapezoids (Lemma \ref{lem:multRule}). This allows us to derive an explicit degree of the $k$-flex locus of $V$ (Proposition \ref{prop:explDegree}), leading to the proof of Theorem \ref{thm:main}.

\section{Dimension and equation of the incidence variety}\label{sec2}

Let us start by giving a precise definition of the $k$-flex locus.

\begin{definition}\label{def:1} Let $V\subset \P^n$ be a subvariety.
 \begin{itemize}
 \item   The \emph{order of contact} between $V$ and a line $L$ of $\P^n$ at some point $p$ is 
\begin{equation*}
\ord_p(V,L)=\dim_{ K}(\mathcal{O}_{L,p}/\iota^*\mathcal{I}_{V}),
\end{equation*} 
where $\mathcal{O}_{L,p}$ is the local ring of $L$ at $p$,
$\mathcal{I}_V$  the ideal sheaf of $V$, and
$\iota\colon L\hookrightarrow \P^n$ the inclusion map.
\item The \textit{osculating order} of $V$
  at $p$ is defined as
\begin{equation*}
\mu_p(V)=\sup_{L} \, \ord_p(V,L),
\end{equation*} 
where the supremum is taken over the lines $L$ of $\P^n$ passing through
$p$. 
\item Let $k\in \N$. The \emph{$k$-flex locus} of $V$ is
  $$
V_k:=\{p\in V\,\,|\,\,\mu_p(V) \ge k\}.
$$
A point $p\in V_k$ is a $k$-flex point. A line with contact order at least $k$ with $V$ at some point is a $k$-flex line. 
\end{itemize}
\end{definition}

We have $\ord_p(V,L)> 0$ if and only if $p\in V\cap L$. We have $\ord_p(V,L)= 1$ if and only if $L$
intersects $V$ transversally at $p$, and $\ord_p(V,L)=+\infty$ if and
only if $L$ is contained in $V$. Moreover, the sequence of inclusions \eqref{eq:stratification} holds.

\subsection{The incidence variety.} Let $n\ge 2$ and denote $\G(1,n)$ the grassmannian of lines in $\P^n$. Let us fix $d\in \N^\times$ a degree and let us denote by $\P^N=\P^{\binom{n+d}{n}-1}$ the space of projective hypersurfaces of $\P^n$ of degree $d$. For $\alpha\in \P^N$ we denote by $V_{\alpha}\subset \P^n$ the corresponding degree $d$ hypersurface.

For all $k\in \N$, we define the incidence variety
$$
\Gamma_k=\{(\alpha,p,L)\in \P^N\times \P^n\times \G(1,n),\,\, \ord_p(V_\alpha,L)\ge k\}.
$$

\begin{proposition}\label{prop:dimSigmak}
The set $\Gamma_k$ is a smooth irreducible projective variety of dimension 
$$
\dim \Gamma_k =N+2n-\min(k,d+1)-1\ge 0.
$$
\end{proposition}

\begin{proof}
The projection $\Gamma_k\to \P^n\times \G(1,n)$ has image 
$$
\Phi=\{(p,L)\in \P^n\times \G(1,n),\, p\in L\},
$$ 
which is itself a $\P^1$-bundle over $\G(1,n)$. 
Since $\G(1,n)$ is an irreducible smooth projective variety of dimension $2n-2$, it follows that $\Phi$ is an irreducible smooth projective variety of dimension $2n-1$. 
If $(p,L)\in \Phi$, then $\ord_p(V_\alpha,L)\ge k$ if and only if the degree $d$ homogeneous polynomial defining $V_\alpha$  lies in the kernel of the restriction map
\begin{equation}\label{eq:mPk}
H^0(\P^n,\cO_{\P^n}(d))\to H^0(\cO_{L}(d)/\mf_{p}^{k}),
\end{equation}
where $\mf_{p}$ stands for the maximal ideal of $\cO_L$ at $p$. The target   $\K$-vector space has dimension $\min(k,d+1)$. The restriction map being surjective, its kernel is a $\K$-vector space of dimension $N+1-\min(k,d+1)$ (which is $>0$ assuming $n\ge 2$). It follows that $\Gamma_k$ is a $\P^{N-\min(k,d+1)}$-bundle over $\Phi$. The claim follows. 
\end{proof}

\subsection{Explicit local equations for $\Gamma_k$.}
Let us choose homogeneous coordinates $x=(x_0,\ldots,x_n)$ for $\P^n$ and consider the general homogeneous polynomial of degree $d$ 
\begin{equation*}
F=\sum_{|I|=d} a_{I}x^{I} \in \K[a,x].
\end{equation*}
The sum is over the vectors $I\in \N^{n+1}$ of length $d$, and 
$a=\{a_{I}\}_{|I|=d}$ stands for the set of $\binom{n+d}{n}$
variables corresponding to the coefficients of $F$. Thus, $F$ is an irreducible polynomial in $\K[a,x]$, bihomogeneous of bidegree $(1,d)$.

\begin{lemma}\label{lem:ordp}
Let $\alpha\in \P^N$ and let $p,q\in \P^n$ be  two distinct points, defining a line $L$. We have the equality
$$
\ord_p(V_\alpha,L)=\ord_t(F(\alpha,p+tq))
$$ (independently of the chosen homogeneous coordinates for $p$ and $q$).
\end{lemma}

\begin{proof} See \cite[Lemma 3.2]{BDSW}. 
\end{proof} 

This lemma motivates to introduce another set of variables $y=(y_0,\ldots,y_n)$ and to consider the polynomials $F_{\ell} \in \K[a,x,y]$, $\ell=0,\dots, d$,     
determined by the Taylor expansion
\begin{equation}
  \label{eq:taylor}
  F(a,x+ty)=\sum_{\ell=0}^{d}F_{\ell}(a,x,y) \frac{t^{\ell}}{\ell !}.
\end{equation}
We have the explicit formula
\begin{equation}\label{eq:explicit_Fell}
  F_\ell(a,x,y)=\sum_{0\le i_1,\ldots, i_\ell\le n} \frac{\partial^{\ell}F}{\partial
    x_{i_{1}}\cdots \partial x_{i_\ell}}(a,x) \, y_{i_1}\cdots y_{i_\ell} .
\end{equation}
In particular, $F_\ell$ is trihomogeneous of tridegree $(1,d-\ell,\ell)$ in $(a,x,y)$. 

\medskip

Let $H\subset \P^n$ be an hyperplane and let $U=\P^n\setminus H$. 
Given $p\in U$, a line passing at $p$ is uniquely determined by its intersection point $q=L\cap H$. Hence, the set $U\times H$ determines a chart of the incidence variety $\Phi\subset \P^n\times\G(1,n)$. Denote by $\Gamma_k(U)\subset \Gamma_k$ the corresponding open subset, projective bundle above $U\times H$. 

\begin{lemma}\label{lem:Sigmak_local_equations}
We have an isomorphism of schemes
\begin{equation*}
\Gamma_k(U)\simeq \{(\alpha,p,q)\in \P^N\times U\times H\, ,\, \, F_0(\alpha,p,q)=\cdots = F_{k-1}(\alpha,p,q)=0\}.
\end{equation*}
\end{lemma}

\begin{proof}
Follows from the previous lemma and the discussion above. This is an isomorphism at the level of schemes since the right hand side subscheme is reduced (use again that it is a projective bundle over $U\times H$). 
\end{proof}

Unfortunately, the polynomials $F_0,\ldots,F_{k-1}$ do not define $\Gamma_k$ globally since they all vanish on the diagonal $p=q$. This will lead to some difficulties to compute the degree of the hyperflex locus. But let us have a look to the properties one can get using the local equations of $\Gamma_k$. 

\section{Hyperflex locus of a general hypersurfaces.}\label{sec3}

In all that follows, we assume that $n+1\le k\le d+1$. By \eqref{eq:stratification}, these are the only interesting values of $k$ for our problem. 

\subsection{About the unicity of the flex line}

\begin{proposition}\label{prop:birational}
The projection $\Gamma_k\to \P^N\times \P^n$ has image 
$$
\Sigma_k=\{(\alpha,p)\in \P^N\times \P^n\, |\, p\in V_{\alpha,k}\}
$$
and the induced surjective morphism $\pi:\Gamma_k\to \Sigma_k$ is birational. 
\end{proposition}

\begin{proof}
Clearly $\pi(\Gamma_k)=\Sigma_k$. Since $\Gamma_k$ is irreducible and $\pi$ is proper, $\Sigma_k$ is an irreducible variety. Hence, it's enough to find one point $(\alpha,p)\in \Sigma_k$ whose fiber is a single reduced point of $\Gamma_k$ to show that $\pi$ is generically one-to-one. 

Let us choose $H=\{x_0=0\}$ and $p_0=(1,0,\ldots,0)\in U=\P^n\setminus{H}$. We can write
\begin{equation}\label{eq:Qjs}
F(a,x)=\sum_{j=0}^d x_0^{d-j} Q_j(a_j,x_1,\ldots,x_n)
\end{equation}
where $Q_j$ is the general homogeneous polynomial of degree $j$ in $(x_1,\ldots,x_n)$ and $a_j$ is its coefficients vector. Notice that $a=(a_0,\ldots,a_{d})$. 
Let $q=(0,y_1,\ldots,y_n)\in H\simeq \P^{n-1}$ represent the lines passing through $p_0$. We get
\begin{equation}\label{eq:taylor_x0}
F(a,p_0+tq)=F(a,1,ty_1,\ldots,ty_n)=\sum_{j=0}^d Q_{j}(a_j,y_1,\ldots,y_n) t^j.
\end{equation}
It follows from Lemma \ref{lem:Sigmak_local_equations} and \eqref{eq:taylor_x0} that the fiber of $\Gamma_k\to \Sigma_k$ above $(\alpha,p_0)\in \P^N\times U$ is isomorphic (scheme-theoretically) to
$$
\Gamma_{k,(\alpha,p_0)}\simeq\{(y_1,\ldots,y_n)\in \P^{n-1},\,\, Q_0(\alpha_0,y)=\cdots=Q_{k-1}(\alpha_{k-1},y)=0\}.
$$
Recall that $k-1\le d$ so that $Q_i\ne 0$ for all $i=0,\ldots,k-1$. Hence, looking for a specialisation $a=\alpha$ for which the fiber over $(\alpha,p_0)$ is a reduced single point amounts to choose $a_0=0$ and to choose $k-1$ homogeneous polynomials in $\K[y_1,\ldots,y_n]$ of respective degrees $1,\ldots k-1$ which have a  a unique reduced common zero $q_0\in \P^{n-1}$. Since $k-1> n-1$, this is possible by the theorems of Bézout and Bertini.
\end{proof}

\begin{remark}\label{rem:contact_order}
If $k\le d$, then $Q_k$ is involved in the expansion \eqref{eq:taylor_x0} and we may choose $\alpha_k$ so that $Q_k(\alpha_k,q_0)\ne 0$. In such a case, the unique $k$-flex line of $V_\alpha$  at $p_0$ has contact order exactly $k$.
If $k=d+1$, then the unique flex line is contained in $V_\alpha$. 
\end{remark}

\subsection{About the non emptyness of the hyperflex locus.}

Recall that we assume that $n+1\le k\le d+1$. 

\begin{proposition}\label{prop:submersion}
The map  $\Gamma_k\to \P^N$ is surjective if and only if $k\le 2n-1$, except for $k = 2n-1 = d+1$.  
\end{proposition}

%

\begin{proof} Let us separate the different cases.

\medskip
\noindent
\emph{Non surjective cases.} If $k\ge 2n$ 
then $\dim \Gamma_k< N$ by Proposition \ref{prop:dimSigmak} so the projection $\Gamma_k\to \P^N$ can not be onto. 
If $k=2n-1=d+1$ then $\dim \Gamma_k=N$. If the map is surjective, then the generic fiber has to be zero-dimensional. But $k=d+1$ implies that $V_k=V_\infty$ is ruled. Hence either $V_k$ is empty or $\dim V_k\ge 1$. In both cases, the generic fiber is not zero-dimensional. Hence $\Gamma_k\to \P^N$ is not onto.

\medskip
\noindent
\emph{Surjective cases.} Since we have inclusions $\Gamma_{k+1}\subset \Gamma_k$ for all $k$, we are reduced to show the surjectivity of $\Gamma_k\to \P^N$ for the cases $k=2n-1\le d$ and $k=2n-2=d+1$.

It's enough to prove local surjectivity since $\P^N$ is irreducible. Hence it is enough to look for a triplet $(V_0,p_0,L_0)\in \Gamma_k$ such that the projection $\Gamma_k\to \P^N$ is a submersion at $(V_0,p_0,L_0)$. 
Let us choose $p_0=(0,\ldots,0)$ in some local coordinates $x_1,\ldots,x_n$ of an affine chart $\A^n\subset \P^n$ and let us consider the line $L_0$ with affine equations $x_1=\cdots=x_{n-1}=0$. One can take local coordinates $(x,y)=(x_1,\ldots,x_n,y_1,\ldots,y_{n-1})$ for the incidence variety $\Phi=\{(p,L)\in \P^n\times\G(1,n)\}$ in a neighborhood of $(p_0,L_0)$, where  the line $L_y$ is parametrized by $$L_y\,:\, t\mapsto (x_1+t y_1,\ldots,x_{n-1}+t y_{n-1},x_n +t).$$

$\bullet$ Suppose that $k= 2n-1\le d$. Then there exists a degree $d$ hypersurface $V_0=V_{\alpha_0}\in \P^N$ with affine equation 
$$
f(x):=\sum_{j=1}^{n} x_n^{2j-2}x_{j}=x_1+ x_2 x_n^2 + x_3 x_n^4 +\cdots +x_{n-1}x_{n}^{2n-4} +x_n^{2n-1}.
$$ 
We get $\ord_{p_0}(L_0,V_0)=\ord_t(f(0,\ldots,0,t))=2n-1$. Thus $(\alpha_0,p_0,L_0)\in \Gamma_{2n-1}$. Let us consider the Taylor expansion
\begin{equation}\label{eq:taylor_exemple}
f(x+ty)=\sum_{i=0}^{2n-1} f_i(x,y)\frac{t^i}{i!}.
\end{equation} 
By Lemma \ref{lem:Sigmak_local_equations}, the fiber of $\Gamma_{2n-1}\to \P^N$ over $\alpha_0$ has (scheme theoretically) local equations
$$
\{f_0(x,y)=\cdots = f_{2n-2}(x,y)=0\}\subset \A^n\times \A^{n-1}
$$
in the neighborhood of $(p_0,L_0)$ and the map $\Gamma_k\to \P^N$ is a local submersion at $(\alpha_0,p_0,L_0)$ if and only if the jacobian matrix of this system of equations has maximal rank $k=2n-1$ at $(x,y)=(0,0)$.
To check this, it is actually enough to compute $f_i(x,y)$ modulo $(x,y)^2$ (square of the maximal ideal).
We find here that
\begin{eqnarray*}
f(x+ty)& =& \sum_{j=1}^{n-1} (x_n+t)^{2j-2}(x_{j}+t y_j)+(x_n+t)^{2n-1} \\
&\equiv & \sum_{j=1}^{n-1} ( t^{2j-2} x_j+ t^{2j-1} y_j) +(2n-1)t^{2n-2} x_n +t^{2n-1}\mod (x,y)^2.
\end{eqnarray*}
Combined with \eqref{eq:taylor_exemple}, we deduce that the differentials of the $f_i$'s at $(0,0)$ are given by (with $j=1\ldots,n-1$)
$$
df_{2j-2}(0,0) = (2j-2)! \,d x_{j} ,\quad df_{2j-1}(0,0)= (2j-1)!\, d y_{j},\quad df_{2n-2}=(2n-1)!\, dx_n.
$$
Up to reorder the variables as $(x_1,y_1,x_2,y_2,\ldots x_{n-1},y_{n-1},x_n)$ it follows that the underlying jacobian matrix 
at $(0,0)$ is square with shape 
$$
{\rm Jac}_{(0,0)}(f_0,\ldots,f_{2n-2})=\begin{pmatrix}
0! & 0 & \cdots & 0 & 0 \\ 
0 & \ddots & \ddots & \vdots & \vdots  \\ 
\vdots & \ddots & \ddots & 0 & \vdots  \\ 
0 & \cdots & 0 & (2n-3)! & 0 \\ 
0 & \cdots & \cdots & 0 & (2n-1)! 
\end{pmatrix},
$$
hence has maximal rank $2n-1$.

$\bullet$ The last case to consider is $k=2n-2$ and $k=d+1$. In such a case, we simply forget  the last monomial of $f$ and consider the truncated polynomial 
$$
\bar{f}(x):=\sum_{j=1}^{n-1} x_n^{2j-2}x_{j}=x_1+ x_2 x_n^2 + x_3 x_n^4 +\cdots +x_{n-1}x_{n}^{2n-4}.
$$
Then $\bar{f}$ has total degree $2n-3=d$ and gives thus the affine equation of a degree $d$ hypersurface $V_0=V_{\alpha_0}$. The fiber of $\Gamma_{2n-2}\to \P^N$ over $\alpha_0$ has (scheme theoretically) local equations
$$
\{f_0(x,y)=\cdots = f_{2n-3}(x,y)=0\}\subset \A^n\times \A^{n-1}
$$
with the same $f_i$'s as above. The underlying jacobian matrix is thus the submatrix
$$
{\rm Jac}_{(0,0)}(f_0,\ldots,f_{2n-3})=\begin{pmatrix}
0! & 0 & \cdots & 0 & 0 \\ 
0 & \ddots & \ddots & \vdots & \vdots  \\ 
\vdots & \ddots & \ddots & 0 & \vdots  \\ 
0 & \cdots & 0 & (2n-3)! & 0 
\end{pmatrix},
$$
which is a $(2n-2)\times(2n-1)$ matrix of maximal rank $2n-2=k$ as required. 
\end{proof}

\begin{remark} This example illutrates the importance to assume that the characteristic of $\K$ is zero or large enough. 
\end{remark}

\subsection{Hyperflex locus of a general hypersurfaces} 

We can now prove the main result of this section, generalizing \cite{BDSW, EisHar}.

\begin{theorem}\label{thm:generic}
Let $n+1\le k\le d+1$ and let $V$ be a \emph{general} hypersurface of degree $d$. Then :

$(1)$  $V_k\subset V$ is a subvariety of pure dimension $\dim V_k=2n-k-1$ (negative dimension is equivalent to empty) except if $k=d+1= 2n-1$ in which case $V_k$ is also empty.  

$(2)$ By a \emph{general} $p\in V_k$, there passes a unique $k$-flex line $L$. This line $L$ satisfies equality $\ord_p(V,L)=k$ if $k\le d$ or is contained in $V$ if $k=d+1$. 
\end{theorem}

\begin{proof}

(1) The $k$-flex locus of an hypersurface $V=V_\alpha$ is (set theoretically) the fiber over $\alpha$ of $\Sigma_k\to \P^N$, hence is a closed subvariety. It follows from Proposition \ref{prop:birational} and Proposition \ref{prop:dimSigmak} that $\Sigma_k$ is irreducible of dimension $\dim \Sigma_k=N+2n-k-1$ and Proposition \ref{prop:birational} ensures that $\Sigma_k\to \P^N$ is surjective if and only if $\Gamma_k\to \P^N$ is surjective, i.e. $k\le 2n-2$ or $k=2n-1 \le d$ thanks to Proposition \ref{prop:submersion}. For such a $k$, we deduce by upper semi-continuity of the fiber's dimension that 
$$\dim V_k\ge \dim \Sigma_k-N=2n-k-1
$$ 
for all $V\in \P^N$, and equality holds for a general $V$. Moreover, the fiber is generically reduced and pure dimensional. For other values of $k$, the map $\Sigma_k\to \P^N$ is not onto and it follows that $V_k$ is empty for a general $V$.

(2) Unicity of the $k$-flex line through a general $k$-flex point follows from Proposition \ref{prop:birational}. If $k=d+1$, any $k$-flex line is obviously contained in $V$ by Bézout's theorem. If $k\le d$, we saw in Remark \ref{rem:contact_order} that there exists a triplet $(V_0,p_0,L_0)$ for which the unique flex line at $p$ has contact order exactly $k$. It follows that $\Gamma_{k+1}\subset \Gamma_k$ is a proper Zariski closed subset. Combined with Proposition \ref{prop:birational}, this shows that $\Gamma_{k+1}$ can not dominate $\Sigma_k$. This precisely means that a general $k$-flex line has contact order exactly $k$ at a general $k$-flex point. 
\end{proof}

\section{Bundle of relative principal parts on $\Phi$}\label{sec4}

Following \cite[Chapter 11]{EisHar}, we will compute the degree of the hyperflex locus in terms of the degree of the top Chern class of some well chosen vector bundle over  the incidence variety
$$
\Phi=\{(p,L)\in \P^n\times \G(1,n),\, p\in L\}.
$$ 
Denote $\beta:\Phi\to \P^n$ the first projection. Let $Z:=\Phi \times_{\G(1,n)} \Phi$ be the fiber product of $\Phi$ with itself according to the proper smooth map $\Phi\to \G(1,n)$, with projections $\pi_1,\pi_2: Z \to \Phi$. Denote $\Delta\subset Z$ the diagonal subscheme and $\cI_\Delta$ its ideal sheaf. The vector bundle on $\Phi$ we are looking for is the so-called bundle of \emph{relative principal parts} defined as 
$$
\cE^{k-1}(d):=\cP^{k-1}_{\Phi/\G(1,n)}(\beta^*\cO_{\P^n}(d)):=\pi_{2*} \left(\pi_1^* (\beta^*\cO_{\P^n}(d))\otimes \cO_Z/\cI_{\Delta}^{k}\right).
$$
Namely, we have the following key result :

\begin{theorem}\label{thm:principal_part}\cite[Thm.11.2]{EisHar}
Let $k,d\ge 1$ be fixed integers.
\begin{enumerate}
\item The sheaf $\cE^{k-1}(d)$ is a rank $k$ vector bundle over $\Phi$ whose fiber at $(p,L)$ is the vector space
$$
\cE^{k-1}(d)_{p,L}= H^0(\cO_{L,p}(d)/\mf_{p}^{k}).
$$
\item We can associate to each $f\in H^0(\P^n,\mathcal{O}_{\P^n}(d))$ a section $\tau_f\in H^0(\Phi,\cE^{k-1}(d))$ whose value at $(p,L)$ is the natural restriction of $f$ to the fiber.
\end{enumerate}
\end{theorem}

\begin{proof}\emph{(sketch, inspired from \cite[Thm.7.2]{EisHar})}.
The sheaf $\cE:=\cE^{k-1}(d)$ is a quasi-coherent sheaf on $\Phi$, and we want to show that it is locally free of rank $k$. Denote $\mathcal{M}:=\beta^*\mathcal{O}_{\P^n}(d)$ and $\mathcal{F}=\pi_1^*\mathcal{M}\otimes \cO_Z/\cI_{\Delta}^{k}$.
Denote $\pi:\Phi\to \G(1,n)$ the natural projection. Let $V\subset \Phi$ be an affine open subvariety with image $U=\pi(V)\subset  \G(1,n)$.  The sheaf $\mathcal{F}$ being supported on the diagonal $\Delta$, we have an isomorphism of $\mathcal{O}(V)$-modules, 
\begin{equation}\label{eq:E(V)}
\cE(V):=\pi_{2*}\mathcal{F}(V)=\mathcal{F}(\pi_2^{-1}(V))\simeq \mathcal{F}(V\times_U V), 
\end{equation}
and we may thus assume that both maps $\pi_1$ and $\pi_2$ are the projections from $V\times_U V$ to $V$. The manifold $\Phi$ being a $\P^1$-bundle over the Grassmannian $\G(1,n)$, we may suppose moreover that $U=\Spec A$ is affine and $V=\Spec(R)$, with $R=A[t]$. In particular, we may suppose $\pi_1$ and $\pi_2$ affine. In such a setting, we have $V\times_U V=\Spec S$ where 
\begin{equation}\label{eq:iso}
S=R\otimes_{A} R= A[t]\otimes_{A} A[t]= A[t_1,t_2],
\end{equation}
where we denote $t_1=t\otimes 1$ and $t_2=1\otimes t$.  
The line bundle $\mathcal{M}$ corresponds to the free rank one $R$-module $M:=\mathcal{M}(V)$ and its pull-back $\pi_1^*(\mathcal{M})$ corresponds to the free rank one $S$-module $M\otimes_A  R$. The $A$-subscheme $\Delta\cap (V\times_U V)$ is defined by the ideal $I\subset S$ which is the kernel of the multiplication map $R\otimes_A R\to R$, that is, $I$ is the principal ideal generated by $t\otimes_A 1-1\otimes_A t=(t_1-t_2)$. Combined with \eqref{eq:E(V)}, we get
$$
\cE(V)=\frac{M\otimes_A  R}{I^k (M\otimes_A  R)},
$$
considered as an $R$-module under the action  $r\cdot(m\otimes_A e):=m\otimes_A re$ (by definition of the direct image $\pi_2^*$). Since $M\otimes_A  R\simeq S$ is a free $S$-module of rank one, we get with \eqref{eq:iso} an isomorphism of $A[t_2]$-modules
\begin{equation}\label{eq:iso2}
\cE(V)\simeq \frac{A[t_1,t_2]}{(t_1-t_2)^k}=A[t_2]\oplus \cdots \oplus (t_1-t_2)^{k-1} A[t_2].
\end{equation}
This shows that $\cE(V)$ is a free $R$-module of rank $k$. This construction is compatible with localization and gluing, hence the sheaf $\cE$ is a rank $k$ vector bundle over $\Phi$, as asserted. 

Let us have a closer look at its fiber. 
We write $\ell\in \G(1,n)$ the point corresponding to the line $L\subset\P^n$. 
A $\K$-rational point $(p,\ell)\in V$ is given by a maximal ideal $\mathfrak{m}=\ker(\psi:R\to \K)$, with residue field $\kappa(p,\ell)$. The projection $\pi:V\to U$ is induced by the natural inclusion $A \subset R=A[t]$ and $\mathfrak{p}=\mathfrak{m}\cap A$ is the maximal ideal of $A$ corresponding to the residue field $\kappa(\ell)$ of $\ell=\pi(p,\ell)\in U$.  Since $\K$ is assumed to be algebraically closed, we have a canonical identification of residue fields
$$
\kappa(p,\ell)=R/\mathfrak{m}= A/\mathfrak{p}=\kappa(\ell)=\K
$$
leading to a canonical identification of $\K$-vector space
$$
S/\mathfrak{m}=S\otimes_R R/\mathfrak{m}=(R\otimes_A R)\otimes_R R/\mathfrak{m}= R\otimes_A R/\mathfrak{m} = R\otimes_A A/\mathfrak{p}=R/\mathfrak{p}.
$$
Thus, the fiber of $\pi_2:V\times_U V\to V$ over $(p,\ell)$ can be canonically identified with the fiber of $F_\ell:=\pi^{-1}(\ell)$ of $\pi:V\to U $ over $\ell$.
Denote $\mathfrak{m}_{p,L}\subset R/\mathfrak{p}$ the image of $I/\mathfrak{m}\subset R/\mathfrak{m}$ under this identification. Then $\mathfrak{m}_{p,\ell}$ is the maximal ideal of $(p,\ell)\in F_\ell$. Denoting $a\in \K$ the residue class of $t$ in $R/\mathfrak{m}=A/\mathfrak{p}$, the class modulo $\mathfrak{m}$ of the generator $t\otimes_A 1-1\otimes_A t$ of $I$ is identified with
$$\overline{t\otimes_A 1-1\otimes_A t}:=t\otimes_A \bar{1}-1\otimes_A \bar{a}=(t-a)\otimes_A \bar{1}=t_1-a
$$
in the ring $R/\mathfrak{p}$. Under the isomorphism \eqref{eq:iso2}, we simply get
$$
\cE_{(p,\ell)}\simeq \frac{A/\mathfrak{p}[t_1]}{(t_1-a)^k}\simeq \frac{\K[t_1]}{(t_1-a)^k}.
$$
In terms of sheaves, we get an isomorphism of $\kappa(p,\ell)$-vector spaces
\begin{equation}\label{eq:iso sur phi}
\cE_{(p,\ell)}=\mathcal{M}\otimes_{\mathcal{O}_{\G(1,n),\ell}} \mathcal{O}_{F_\ell,(p,\ell)}/\mathfrak{m}_{p,\ell}^k.
\end{equation}
The restriction to the open set $V$ of the map $\beta:\Phi\to \P^n$ maps the fiber $F_\ell$ isomorphically to the (affine) line $L_0:=L\cap \beta(V)\subset \P^n$
defined by $\ell$ and $\mathfrak{m}_{p,\ell}$ maps through $\beta$ to the maximal ideal $\mathfrak{m}_{p}$ of $p\in L_0$. Reminding that $\mathcal{M}=\beta^*(\mathcal{O}_{\P^n}(d))$, pushing forward through $\beta_*$ the right hand term of \eqref{eq:iso sur phi} leads to the canonical identification of $\kappa(p,L)$-vector space
$$
\cE_{(p,L)}\simeq \mathcal{O}_{\P^n}(d)\otimes_{\mathcal{O}_{\P^n,p}} \mathcal{O}_{L_0,p}/\mathfrak{m}_{p}^k = \mathcal{O}_{L_0,p}(d)/\mathfrak{m}_{p}^k,
$$
as required. Given a section $f\in H^0(\beta(V),\mathcal{O}_{\P^n}(d))$,  the image of $\beta^* f\otimes_A 1$ in $M\otimes_A  R/I^k (M\otimes_A  R)$ will give an element whose value at $(p,\ell)$ will coincides with the image of $f$  in $\mathcal{O}_{L_0,p}(d)/\mathfrak{m}_{p}^k$. These sections will naturally glue together when considering a covering of $\Phi$ by various open sets $V$. 


\end{proof}

For $\alpha\in \P^N$, let $f_\alpha\in H^0(\P^n,\mathcal{O}_{\P^n}(d))\simeq \mathbb{A}^{N+1}$ be an arbitrary homogeneous polynomial defining $V_\alpha$ and denote for short $\tau_\alpha=\tau_{f_\alpha}$, with notations of point (2) in Theorem \ref{thm:principal_part}.

\begin{corollary}\label{cor:sigmaF_reduced} \ 
\begin{enumerate}
\item We have $\Gamma_{k}=\{(\alpha,p,L)\in \P^N\times \Phi,\,\, \tau_{\alpha}(p,L)=0\}$.
\item For a general $\alpha$, the subscheme $\{\tau_{\alpha}=0\}\subset \Phi$ is reduced and maps birationally to the $k$-flex locus $V_{\alpha,k}$  of $V_\alpha$ via the projection $\Phi\to \P^n$.
\end{enumerate}
\end{corollary}

\begin{proof} 
Given $(p,L)\in \Phi$, we have $(\alpha,p,L)\in \Gamma_k$ if and only if $f_\alpha$ lies in the kernel of the restriction map
$$
H^0(\P^n,\cO_{\P^n}(d))\to H^0(\cO_{L,p}(d)/\mf_{p}^{k})
$$
that is if and only if $\tau_{\alpha}(p,L)=0$ by Theorem \ref{thm:principal_part}. For the second point, consider the commutative diagram
\begin{equation*}
  \begin{tikzpicture}[scale=0.9]
            \node (A) at (0,0) [inner sep=4pt] {$\Gamma_k$};
            \node (B) at (3,0) [inner sep=4pt] {$\Sigma_k$};
            \node (C) at (1.5,-1.5) [inner sep=4pt] {$\mathbb P^N$};
            \draw[->] (A) --node[above]{\tiny $\pi$} (B);
            \draw[->] (A) --node[left]{\tiny $p_1$} (C);
            \draw[->] (B) --node[right]{\tiny $p_2$} (C);
        \end{tikzpicture}
         \end{equation*}
induced by the natural projection. The fibers of $p_1$ and $p_2$ above a general $\alpha$ are reduced and respectively isomorphic to 
$$
p_1^{-1}(\alpha)\simeq \{\tau_{\alpha}=0\}\quad {\rm and}\quad p_2^{-1}(\alpha)\simeq V_{\alpha,k}. 
$$
On the other hand, since the morphism $\pi$ is birational by Proposition \ref{prop:birational}, the general fibers $p_1^{-1}(\alpha)$ and $p_2^{-1}(\alpha)$ are birational too (meaning that there is a bijection between their irreducible components, the components being mapped birationaly to each other). Indeed, there is a non empty open set $U\subset \Gamma_k$ such that $\pi_{|U}$ is an isomorphism onto its image. Since the fibers of $p_1$ are generically equidimensional, it's enough to show that the general fiber of $p_1$ has no component in the closed subvariety $\Gamma'=\Gamma_k\setminus U\subset \Gamma_k$. If this is not the case, we would have $\dim \Gamma'= \dim p_1^{-1}(\alpha)+N$ for $\alpha$ in an open set of $\P^N$, which would lead to $\dim \Gamma'=\dim \Gamma_k$, a contradiction. 
\end{proof}

\section{Degree of the hyperflex locus}\label{sec5}

We briefly recall the notion of Chern classes of a vector bundle on a smooth variety $X$ and some useful properties that we will use in order to compute the  degree of the hyperflex locus. We refer the reader to \cite{EisHar, Ful, hart} for classical references on intersection theory. 

We denote by $A^i(X)$ the Chow group of cycles of $X$ of codimension $i$ modulo rational equivalence. The Chow ring of $X$ is $$A(X)=A^0(X)\oplus \cdots  \oplus A^{\dim(X)},$$ the multiplication being given by the intersection product. In particular, if two irreducible subvarieties $Y,Z\subset X$ of respective codimension $i$ and $j$ intersect each other transversally, the product is $[Y]\cdot [Z]:=[Y\cap Z]\in A^{i+j}(X)$. 

\begin{theorem}\label{thmChernClasses}\cite[Theorem 5.3]{EisHar}, \cite[Appendix A, Section 3]{hart}
There is a unique way to assign to a vector bundle $\mathcal E$ of rank $r$ on $X$ a class $c(\mathcal E)=1+c_1(\mathcal E)+\cdots +c_t(\mathcal E) \in A(X)$ such that 
\begin{enumerate}
\item if $\mathcal L$ is a line bundle on $X$, then $c(\mathcal L)=1+c_1(\mathcal L)$, where $c_1(\mathcal L)\in A^1(X)$ is the class of the divisor of zeros minus the divisor of the poles of any rational section of $\mathcal L$;
\item \label{itthmCC:2}if $\tau_0,\dots,\tau_{r-i}$ are global sections of $\mathcal E$, and if the degeneracy  locus $D$ where these sections are linearly dependent has codimension $i$, then $c_i(\mathcal E)=[D]\in A^i(X)$;
\item \label{itthmCC:3}(Whitney's formula)  if
\[
0\rightarrow \mathcal E\rightarrow \mathcal F\rightarrow \mathcal G\rightarrow 0,
\]
is a short exact sequence of vector bundles on $X$, 
then $c(\mathcal F)=c(\mathcal E)\cdot c(\mathcal G)$;
\item if $\varphi: Y\rightarrow X$ is a morphism of smooth varieties, then $\varphi^\ast(c(\mathcal E))=c(\varphi^\ast(\mathcal E))$; 
\item (Splitting principle) any identity among Chern classes of bundles that
is true for bundles that are direct sums of line bundles is true in general.
\end{enumerate}
\end{theorem}
\begin{definition}
Using the notations of Theorem \ref{thmChernClasses}, we call $c(\mathcal E)$ the \emph{total Chern class}  of $\mathcal E$, and for all $i$ we call $c_i(\mathcal E)\in A^i(X)$ the $i$-th Chern class of $\mathcal E$.
\end{definition}

Combining Whitney's formula and the Splitting Principle, one immediately obtains, for instance, that if $\mathcal E$ is a vector bundle of rank $r$, then $c_i(\mathcal E)=0$ for all $i>r$. We call $c_r(\mathcal E)$ the \emph{top Chern class} of $\mathcal E$.

We now prove that the degree of the hyperflex locus of the variety $V$ is easily deduced from the top Chern class of the bundle of relative principal parts on $\Phi$. Recall that for $X$ smooth and proper, the degreee map $$\deg :A(X)\to \Z$$ is defined by the usual degree of $0$-cycles if $\alpha\in A^{\dim(X)}$ while $\deg(\alpha)=0$ if  $\alpha\in A^{i}(X)$, $i<\dim(X)$, see \cite[Definition 1.4]{Ful}.

\begin{proposition}\label{cor:degVk+1}
Denote $\zeta=\beta^*(\omega)\in A^1(\Phi)$ the pull-back of the hyperplane class $\omega$ of $\P^n$.  For a general $V\in \P^N$, we have 
$$
\deg(V_{k})=\deg(c_{k}(\cE^{k-1}(d))\cdot \zeta^{2n-1-k}). 
$$
\end{proposition}

\begin{proof}
Let $\alpha\in \P^N$ be the class of a general degree $d$ polynomial and denote $V=V_\alpha$. We have $\dim V_{k}=2n-1-k$ by Theorem \ref{thm:generic} (1), so that $\deg(V_{k})=[V_{k}]\cdot \omega^{2n-1-k}$.  From point (2) of Corollary \ref{cor:sigmaF_reduced} we have $[V_{k}]=\pi_*[\tau_{\alpha}=0]$ and the projection formula \cite[Proposition 2.5]{Ful} gives
$$
\deg(V_{k})=\deg([\tau_\alpha=0]\cdot
\zeta^{2n-1-k}).
$$

Since $\tau_{\alpha}$ is a global section of the rank $k$ vector bundle $\cE^{k-1}(d)$,  by Theorem \ref{thmChernClasses} item \eqref{itthmCC:2} the Chow class of $\tau_{\alpha}=0$ coincides with the top Chern class $c_{k}(\cE^{k-1}(d))$.
Notice that $\codim (\tau_{\alpha}=0)+2n-1-k=2n-1=\dim \Phi$ so we are  computing the degree of a zero-cycle on $\Phi$, as required. 
\end{proof}

In order to compute $c_{k}(\cE^{k-1}(d))$, we will first compute $c(\cE^{k-1}(d))$ using that total Chern classes satisfy Whitney's formula together with the following key property satisfied by the  bundle of relative principal parts :

\begin{proposition}\label{prop:exact_sequence}
 \cite[Thm.11.2 (d)]{EisHar} We have $\cE^0(d)=\beta^*(\cO_{\P^n}(d))$. For $k\ge 1$, we have an exact sequence of vector bundles over $\Phi$
$$
0\,\rightarrow\, \cE^0 (d)\otimes \Sym^{k}(\Omega_{\Phi/\G(1,n)}) \,\rightarrow\, \cE^k(d) \,\rightarrow\, \cE^{k-1}(d) \,\rightarrow\, 0.
$$
\end{proposition}

This will allow us to compute $c(\cE^{k-1}(d))\in A(\Phi)$ by induction. In order to compute in $A(\Phi)$, we use Schubert calculus on the Grassmannian, as introduced for instance in \cite[Chapter 4]{EisHar}. 

Given $a,b\in \N$ such that $0\le b \le a \le n-1$, we denote $\sigma_{a,b} \in A^{a+b}(\G(1,n))$ the corresponding Schubert class, denoting for short $\sigma_a=\sigma_{a,0}$. Each $\sigma_{a,b}$ is the class of a subvariety of $\G(1,n)$ given by those lines $L\subset \P^n$ whose intersections with each linear projective subspace $W\subset \P^n$ of a given projective flag of $\P^n$ have dimension at least $0$ or $1$, depending on the values of $a,b$ and $\dim W$ \cite[Chapter 4.1]{EisHar}. This definition does not depend on the chosen flag. The Schubert classes form a basis of the free abelian groupe $A(\G(1,n))$ and there are explicit formulas for the  product of two Schubert classes \cite[Corollaries 4.7 and 4.8]{EisHar}. Although we won't use this directly, let us mention for completeness that the Chow rings of $\Phi$ and $\G(1,n)$ are related by
\begin{equation}\label{eq:ChowRingPhi}
A(\Phi)=A(\G(1,n))[\zeta]/(\zeta^2-\sigma_1 \zeta+\sigma_{1,1}),
\end{equation}
see \cite[Section 9.3.1]{EisHar}.
In what follows, we still denote by $\sigma_{a,b}\in A^{a+b}(\Phi)$ the pull-back of $\sigma_{a,b}$ by $\Phi\to \G(1,n)$ for readability. Also, given $0\le \ell\le k-1$, we denote for short
\begin{equation}\label{eq:rellk}
\mu_{\ell}(d,k)=d(k-1)!\left(\sum_{1\leq i_1<\cdots<i_{\ell}\leq k-1}\frac{(d-2i_1)\cdots(d-2i_\ell)}{i_1\cdots i_\ell}\right)
\end{equation}
with convention $\mu_0(d,k)=d(k-1)!$. Note that $\mu_{\ell}(d,k)\in \N$ for all $\ell$. We get:

\begin{theorem}\label{thm:degree_in_terms _of_Chern_classes} Let $V\subset \P^n$ be a general hypersurface of degree $d$. 
 Let $n+1 \le k \le d+1$ and $k\le 2n-1$. Then 
\begin{equation}\label{eq:degree}
\deg(V_{k})= \sum_{\ell=0}^{k-n} \mu_{\ell}(d,k)\, \deg\left(\sigma_{2n-k-1+\ell} \,\sigma_1^{k-1-\ell}\right)_{\G(1,n)}  .
\end{equation}
\end{theorem}
\begin{proof} If $k> 2n-1$ then $V_k$ is empty, by Theorem \ref{thm:generic}. Otherwise, by Whitney's formula (Theorem \ref{thmChernClasses} item \eqref{itthmCC:3}) and Proposition \ref{prop:exact_sequence}, we get by induction that the total Chern class of $\cE^{k-1}(d)$ satisfies
$$
c(\cE^{k-1}(d))=\prod_{j=0}^{k-1} c\left(\beta^*(\cO_{\P^n}(d))\otimes \Sym^{j}(\Omega_{\Phi/\G(1,n)})\right)\in A(\Phi).
$$
The bundle $\Omega_{\Phi/\G(1,n)}$ is a line bundle over $\Phi$, and we thus have
$$
c\left(\Sym^{j}(\Omega_{\Phi/\G(1,n)})\right)=1+j c_1(\Omega_{\Phi/\G(1,n)}).
$$
It turns out that we can compute this first Chern class. Namely, the dual of $\Omega_{\Phi/\G(1,n)}$ is the relative tangent bundle $\mathcal{T}_{\Phi/\G(1,n)}$ of $\Phi\to \G(1,n)$ and \cite[Thm.11.4]{EisHar} gives
$$
c_1(\Omega_{\Phi/\G(1,n)})=-c_1(\mathcal{T}_{\Phi/\G(1,n)})=\sigma_1-2\zeta \,\,\in\, A^1(\Phi).
$$ 
Using $c_1(\beta^*(\cO_{\P^n}(d)))=d\zeta$, this leads to
$$
c(\cE^{k-1}(d))=\prod_{j=0}^{k-1} (1+j\sigma_1+(d-2j)\zeta).
$$
By Proposition \ref{cor:degVk+1},  we need to multiply the homogeneous part $c_{k}(\cE^{k-1}(d))$ of $c(\cE^{k-1}(d))$ by $ \zeta^{2n-1-k}$, leading to:
\begin{equation}\label{eq:degV_k}
\deg(V_{k})=\deg\left( \zeta^{2n-1-k}\cdot \prod_{j=0}^{k-1} (j  \sigma_1+(d-2j)\zeta)\right)_\Phi.
\end{equation}
Using symmetric elementary functions, a straightforward computations shows that
\begin{equation}\label{eq:mukl}
\prod_{j=0}^{k-1} (j  \sigma_1+(d-2j)\zeta)=\sum_{\ell=0}^{k-1} \mu_{\ell}(d,k)\zeta^{\ell+1} \sigma_1^{k-1-\ell},
\end{equation}
the second equality using notations \eqref{eq:rellk}. Combined with \eqref{eq:degV_k}, and considering that $\sigma_1^t\zeta^w=0$ for every $w>n$ since $\zeta$ is the pull-back of the hyperplane class on $\P^n$, we get 
\begin{equation}\label{eq:degExp}
\deg(V_k)= \deg\left(\sum_{\ell=0}^{k-n} \mu_{\ell}(d,k)\zeta^{2n-k+\ell} \sigma_1^{k-1-\ell}\right)_\Phi.
\end{equation}
There remains to express the degrees of the non-zero monomials $\zeta^{2n-k+\ell}\sigma_1^{k-1-\ell}$ appearing in \eqref{eq:degExp} in terms of the degrees of the Schubert classes. Instead of using \eqref{eq:ChowRingPhi}, we rather use the fact that $\Phi=\mathbb P\mathcal S$ is the projectivization of the universal subbundle $\mathcal S\to \G(1,n)$ of the Grassmannian \cite[Section 3.2.3]{EisHar}. Denoting by  $\gamma$ the projection  $\Phi\to  \G(1,n)$, we get that 
 \[
 \deg(\zeta^{a}\sigma_1^{b})_\Phi=\deg \gamma_\ast(\zeta^{a}\sigma_1^{b})_{\G(1,n)}=\deg(s_{a-1}(\mathcal S)\sigma_1^{b})_{\G(1,n)},
 \]
whenever $a+b=\dim \Phi=2n-1$, where $s_a(\mathcal S)\in A^a(\G(1,n))$ stands for the $a$-th Segre class of $\mathcal{S}$ \cite[Definition 10.1]{EisHar}. The total Segre class $s(\mathcal S):=\sum_{a\ge 0} s_a(\mathcal S)$ satisfies
$$
s(\mathcal S)=\frac{1}{c(\mathcal S)}=1+\sigma_1+\sigma_2+\cdots+\sigma_{n-1},
$$
the first equality by \cite[Proposition 10.3]{EisHar}, and the last equality by the last statement of \cite[Section 5.6.2]{EisHar} combined with \cite[Corollary 4.10]{EisHar}. It follows that $s_{a}(\mathcal S)=\sigma_{a}$ for all $a\le n-1$, with convention $\sigma_0=1$. 
Summing up, we deduce that
\[
\deg(\zeta^{2n-k+\ell}\sigma_1^{k-1-\ell})_\Phi=\deg(\sigma_{2n-k-1+\ell}\sigma_1^{k-1-\ell})_{\G(1,n)}
\]
for all $\ell\le k-n$, which combined with \eqref{eq:degExp} leads to the desired formula. 
\end{proof}

We now explicit \eqref{eq:degree} by giving a formula for $\deg(\sigma_{a}\sigma_1^m)$. Let us first recall the well-known Pieri's formulas \cite[Proposition 4.9]{EisHar} in the special case $\sigma_{a,b}\sigma_1$ for the grassmannian of lines $\G(1,n)$. Recall that $\sigma_a=\sigma_{a,0}$.
\begin{proposition}[Pieri's formula]\label{prop:pieriforus}
Let $a,b\in \N$ such that $0\le b \le a \le n-1$. Then
\[
\sigma_{a,b}\sigma_1=\begin{cases}\sigma_{a+1,b}+\sigma_{a,b+1} \quad\,\, \text{ if } \,b+1\leq a\leq n-2\\
\sigma_{a,b+1} \qquad \qquad \quad  \text{ if } \, b+1\leq a=n-1 \\
\sigma_{a+1,b} \qquad \quad \quad \,\,\,\,\,\, \text{ if } \, b+1>a \,\,{\text and}\,\, a\leq n-2\\
0 \quad \qquad \quad \qquad \,\,\,\,\,\, \text{ if } \, b=a=n-1. \end{cases}
\]
\end{proposition}

The other tool we use in order to get an explicit expression for $\deg(V_k)$ from \eqref{eq:degree} is the Catalan's trapezoid of order $a$, which generalize the Catalan's triangle. We recall here the definitions and some properties, for more details on Catalan's triangle and  trapezoids see for instance \cite{Reuv,Thomas}.

\begin{definition}\label{def:CatTrap}
Let $C_a(u, v)$ denote the $(u, v)$ entry of the Catalan’s trapezoid of order $a$, $a\in \mathbb N\setminus\{0\}$. 
The value of $C_a(u,v)$ is given by the following closed formula
\begin{equation}\label{eq:CatalanBinomial}
C_a(u,v)=\begin{cases}
    \binom{u+v}{v} \quad \text{ if } 0\leq v <a\\
    \binom{u+v}{v}-\binom{u+v}{v-a} \quad \text{ if } a\leq v \leq u+a-1\\
    0 \text{ otherwise}
\end{cases}.
\end{equation}

Equivalently, one can recursively define $C_a(u,v)$ in the following way. 
We define $C_a(u, 0) = 1$ for all $a$, 
and  $C_a(0, v) = 1$ for $0 \leq v \leq a-1$, and $C_a(u, v) = 0$  for all values $(u,v)$ such that $v>u+a-1$.
For all the other values of $u,v$,  $C_a(u,v)$ is defined by the following  recursive rule:
\begin{equation}\label{eq:defCatTrap}
C_a(u, v) = C_a(u-1, v) + C_a(u, v-1).
\end{equation}
For $a=1$, one obtains the well-known  \emph{Catalan's triangle} $C_1(u, v) = C(u, v)$.
\end{definition}

\begin{example}
The name Catalan's {\em trapezoid} becomes clear when we plot the numbers $C_a(u,v)$ in a table, with rows indexed by $u$ and columns by $v$.  Below, you find the tables for $a=2,3$. For sake of readability, if $C_a(u,v)=0$, in the following tables the  entry $u,v$ is empty.
\[
\begin{array}{|c||c|c|c|c|c|c|c|}
\hline
a=2& 0 &1& 2 & 3& 4 & 5 &6 \\
\hline \hline
0 &\mathbf{1}&\mathbf{1}&&&&&\\ \hline
1&\mathbf{1}& 2 & 2&&&&\\\hline
2&\mathbf{1} & 3 &5 &5&&&\\\hline
3&\mathbf{1}& 4 & 9 &14 &14 &&\\\hline
4&\mathbf{1}&5 &14 &28 &42 &42&\\\hline
5&\mathbf{1}& 6 &20 &48 &90 &132&132\\\hline
\end{array}
\quad \begin{array}{|c||c|c|c|c|c|c|c|c|}
\hline
a=3& 0 &1& 2 & 3& 4 & 5 &6 &7\\
\hline \hline
0 &\mathbf{1}&\mathbf{1}&\mathbf{1}&&&&&\\ \hline
1&\mathbf{1}& 2 & 3&3&&&&\\\hline
2&\mathbf{1} & 3 &6 &9&9&&&\\\hline
3&\mathbf{1}& 4 & 10 &19 &28 &28&&\\\hline
4&\mathbf{1}&5 &15 &34 &62 &90&90&\\\hline
5&\mathbf{1}& 6 &21 &55 &117 &207&297&297\\\hline
\end{array}
\]
\end{example}

By Pieri's formula and the values of the Catalan's trapezoid, we can now prove the following. Remember that  if $a\geq n$ or $b>a$ then $\sigma_{a,b}=0$ in $A(\G(1,n))$.

\begin{lemma} \label{lem:multRule}Consider $\sigma_{a,b}, \sigma_1 \in A(\G(1,n))$. Then
\[
\sigma_{a,b} \,\sigma_{1}^m=\sum_{\substack{0\leq i\leq m \\ b+i\leq a+m-i\leq n-1} } C_{a-b+1}(m-i,i)\,\sigma_{a+m-i,b+i},
\]
where $C_{a-b+1}(m-i,i)$ is the positive integer in the $(m-i)$-th row and $i$-th column of the Catalan's trapezoid of order $a-b+1$.
\end{lemma}
\begin{proof}
In order to compute $\sigma_{a,b} \sigma_{1}^m$  we repeatedly apply Pieri's formula. Remember that  if $a\geq n$ or $b>a$ then $\sigma_{a,b}=0$ in $A(\G(1,n))$.
For $m=0$, the thesis is trivially true since for all $a\geq b$, $C_{a-b+1}(0,0)=1$.
Assume now the thesis to be true for $\sigma_{a,b}\sigma_1^{t}$, for every $t<m$. Then, for $\sigma_{a,b}\sigma_1^{m}$ we obtain
\[
\sigma_{a,b}\sigma_1^{m}=(\sigma_{a,b}\sigma_1^{m-1})\sigma_1=\sum_{\substack{0\leq i\leq m-1 \\  b+i\leq a+m-1-i\leq n-1} } C_{a-b+1}(m-1-i,i)\sigma_{a+m-1-i,b+i}\sigma_1.
\]
Using Pieri's formula  as stated in Proposition \ref{prop:pieriforus} on every summand, we obtain exactly the thesis, by Definition \ref{def:CatTrap} and in particular by the relation \eqref{eq:defCatTrap}.
\end{proof}

\begin{proposition}\label{prop:explDegree}
Let $n+1 \le k \le 2n-1$. For every $\ell\in \{0,\dots, k-n\}$ we obtain
\[
\sigma_{2n-k-1+\ell}\sigma_{1}^{k-1-\ell}= C_1(n-1,k-n-\ell)
 \sigma_{n-1,n-1}.
\]
Consequently, assuming $k\le d+1$, we get (using notations \eqref{eq:rellk}):
\begin{equation}\label{eq:final_degree}
\deg(V_{k})= N_k(n,d):= \sum_{\ell=0}^{k-n} \mu_{\ell}(d,k)C_1(n-1,k-n-\ell).
\end{equation}
\end{proposition}
\begin{proof}
We apply Lemma \ref{lem:multRule} with $a=2n-k-1+\ell$,  $b=0$ and $m=k-1-\ell$ to compute $\sigma_{2n-k-1+\ell}\sigma_{1}^{k-1-\ell}$. We observe that the only value of $i\in \{0,\dots,k-1-\ell\}$ such that $i\leq 2n-1-k+\ell+k-\ell-1-i\leq n-1$ is $i=n-1$, for which both inequalities are equalities.  This proves the first point. Furthermore, observing that for every $\ell=0,\dots,k-n-1$ we have $2n-k+\ell\leq n-1$, we get from \eqref{eq:CatalanBinomial}  and from the binomial identity $\binom{r}{s}=\binom{r}{r-s}$
\begin{multline*}
C_{2n-k+\ell}(k-n-\ell,n-1)=\binom{k-\ell-1}{n-1}-\binom{k-\ell-1}{k-n-\ell-1}=\\=\binom{k-\ell-1}{k-\ell-n}-\binom{k-\ell-1}{k-n-\ell-1}=C_1(n-1,k-n-\ell).
\end{multline*}
For $\ell=k-n$, we get $C_n(0,n-1)=1=C_1(n-1,0)$
.  Since  $\sigma_{n-1,n-1}$ generates the degree $0$ part of the Chow ring $A(\G(1,n))$ and is the class of a point, we get $\deg(\sigma_{n-1,n-1})_{\G(1,n)}=1 $. The statement for $\deg(V_{k})$ thus follows from the first point of the proposition together with Theorem \ref{thm:degree_in_terms _of_Chern_classes}.
\end{proof}

\vspace{0.2cm}
\noindent
\textit{Proof of Theorem \ref{thm:main}.}
Items \eqref{it1thmMain} and \eqref{it3thmMain} of Theorem \ref{thm:main} follow from Theorem \ref{thm:generic}. Concerning item \eqref{it2thmMain} (degree formula), we simply write \eqref{eq:final_degree} in a more convenient form. Recall that the integers $\lambda_m(d,k)\in \N$ in Theorem \ref{thm:main} are defined as the coefficients of the polynomial 
\begin{equation}\label{eq:coeff_cm}
P_{d,k}(X)=d\prod_{j=1}^{k-1}\left(jX+(d-2j)\right)=:\sum_{m=0}^{k-1} \lambda_m(d,k) X^m \,\in \N[X].
\end{equation}
Looking at \eqref{eq:rellk} or \eqref{eq:mukl}, we see that the reciprocal polynomial of $P_{d,k}$ has coefficients $\mu_0(d,k),\ldots\mu_{k-1}(d,k)$. It follows that
$$
\mu_{\ell}(d,k)=\lambda_{k-1-\ell}(d,k) \quad \ell=0,\ldots,k-1.
$$
Applying the change of index $m=k-\ell-1$ in \eqref{eq:final_degree} and using $\binom{m}{m-n}=\binom{m}{n}$ with convention $\binom{a-1}{a}=\binom{a}{-1}=0$ for any $a\in \N$, we get the desired formula
\begin{equation}\label{eq:final_degree_really}
N_k(n,d)=\sum_{m=n-1}^{k-1} \lambda_{m}(d,k)\left(\binom{m}{n-1}-\binom{m}{n}\right).
\end{equation}
$\hfill\square$

\vspace{0.2cm}
\noindent
\textit{Proof of Corollary \ref{cor:1}.} The first claim of Corollary \ref{cor:1} follows from Theorem \ref{thm:main} together with the fact that the biggest ruled subvariety $V_{\infty}$ coincides with $V_{d+1}$ thanks to \eqref{eq:stratification}. 
For the second claim of Corollary \ref{cor:1}, we simply remark that if $\deg\,V=2n-3$, then considering $k=2n-2=d+1$, we get that $V_{k}=V_\infty$ has dimension $1$ and is made of the lines contained in $V$. Hence, the number of such lines is $\deg(V_{2n-2})$. $\hfill\square$

%


\vspace{0.2cm}
\noindent
As another corollary, we recover the following result, which is \cite[Corollary 1.2]{BDSW}.
\begin{corollary}\label{cor:deg_flex_locus}
Let $V\subset \P^n$ be a general hypersurface of degree $d\ge n$. The flex locus $V_{n+1}$ is a codimension one subvariety of $V$ of degree
\begin{equation}\label{eq:deg_flex}
\deg V_{n+1}=d^2\cdot\sum_{i=1}^n \frac{n!}{i}-d(n+1)!
\end{equation}
\end{corollary}
\begin{proof}
Setting $k=n+1$ in \eqref{eq:final_degree_really}, we get 
$$
\deg(V_{n+1})=\lambda_{n-1}(d,n+1)+(n-1)\lambda_{n}(d,n+1).
$$
Looking at \eqref{eq:coeff_cm} with $k=n+1$, we compute
$$
P_{d,n+1}(X)=dn!\left(X^n+ X^{n-1}\sum_{i=1}^n\frac{d-2i}{i}+\cdots\right).
$$
Hence $\lambda_{n}(d,n+1)=dn!$ while $$\lambda_{n-1}(d,n+1)=dn!\sum_{i=1}^n\frac{d-2i}{i}=d^2 \sum_{i=1}^n \frac{n!}{i} - 2 dn \,n!\,.$$
The claimed formula follows.
\end{proof}

\begin{remark}\label{rem:flex_vs_hyperflex}
The writing of $\deg V_{n+1}$ as in \eqref{eq:deg_flex}  highlights some geometric fetures of $V_{n+1}$. Indeed, the first term in the formula comes from the intersection of $V=(F=0)$ with the zero locus of a multivariate resultant induced by Lemma  \ref{lem:Sigmak_local_equations} applied with $k=n+1$, while the second term comes from an excess of intersection at infinity due to an extra factor of the resultant modulo $F$. The authors in \cite{BDSW} succeed to identify this extra factor, and prove in such a way that $V_{n+1}=(F=G=0)$ is a complete intersection, giving moreover an explicit formula for the (non unique) homogeneous polynomial $G$. We may wonder if the hyperflex locus is also a complete intersection for higher values of $k$.  The following  example will show that this is not the case in general. 
 \end{remark}
\begin{example}
Thanks to Theorem \ref{thm:main}, we explicitely compute the values of the degree of the $k$-flex locus of a generic hypersurface $V$ of degree $d$ in $\mathbb P^n$, for $n\leq 6$ and the meaningful values of $k$. Hence, in the following tables the reader finds $\deg(V_k)=N_k(n,d)$, for $n\leq 6$ and $n+1\leq k\leq 2n-1$. 

\vspace{0.4cm}
\noindent
\adjustbox{max width=\textwidth}{
    \centering
\begin{tabular}{|c||*{3}{c|}}
\hline\hline
\backslashbox{$n$}{$k$} & $n+1$ & $n+2$  & $n+3$ 
\\
\hline
\raisebox{0em}[1.5em][0.5em]{$2$} & $3 d^{2}-6 d $ & $0$ & $0$ 
\\ \hline
\raisebox{0em}[1.5em][0.5em]{$3$} & $11 d^{2}-24 d $ & $35 d^{3}-200 d^{2}+240 d $ & $0$
\\ \hline
\raisebox{0em}[1.5em][0.5em]{$4$} & $50 d^{2}-120 d $ & $225 d^{3}-1370 d^{2}+1800 d $ & $735 d^{4}-8120 d^{3}+26460 d^{2}-25200 d $
\\ \hline
\raisebox{0em}[1.5em][0.5em]{ $5$} & $274 d^{2}-720 d $ & $1624 d^{3}-10584 d^{2}+15120 d $ & $6769 d^{4}-78792 d^{3}+274428 d^{2}-282240 d $ 
\\  \hline
\raisebox{0em}[1.5em][0.5em]{ $6$} & $1764 d^{2}-5040 d $ & $13132 d^{3}-91476 d^{2}+141120 d $ & $67284 d^{4}-826868 d^{3}+3068352 d^{2}-3386880 d $ 
\\
\hline\hline
\end{tabular}
}

\vskip.5cm
\noindent
\adjustbox{max width=\textwidth}{
    \centering
\begin{tabular}{|c||*{2}{c|}}
\hline\hline
\backslashbox{$n$}{$k$} & $n+4$ & $n+5$
\\
\hline
\raisebox{0em}[1.5em][0.5em]{$5$} & $22449 d^{5}-403704 d^{4}+2480604 d^{3}-6136704 d^{2}+5080320 d $ & $0$ 
\\ \hline
\raisebox{0em}[1.5em][0.5em]{$6$} & $269325 d^{5}-5065760 d^{4}+32835600 d^{3}-86232384 d^{2}+76204800 d $ & $902055 d^{6}-23918510 d^{5}+235466000 d^{4}-1071300384 d^{3}+2232014400 d^{2}-1676505600 d$ 
\\ \hline
\end{tabular}
}

\vskip.5cm

 
Regarding complete intersection issues, the first non trivial example is the $5$-flex locus of a general surface $S\subset \P^3$ of degree $d\ge 5$. It follows from this table (see also \cite[Section 11.1.3]{EisHar}) that $S_5$ is a zero-dimensional reduced subvariety $S_5\subset S$ of degree
$$
N_5(d,5)=35 d^3-200 d^2+240 d=5d(d-4)(7d-12).
$$
By Bézout's theorem, this factorization does not exclude that $S_5$ is a complete intersection in $\P^3$ for all $d$.  However, for $n=5$, $k=8$ and $d=53$, we find the prime factorization
$$
N_8(53,5)=42436258837= 7 \times 53 \times 114383447.
$$ 
Since $\codim V_8=4$ here, this excludes definitely that $V_8$ is a complete intersection.
\end{example}

\section*{Acknowledgments}
The authors acknowledge the IEA (International Emerging Action) project PAriAlPP
(Probl\`emes sur l’Arithm\'etique et l’Alg\`ebre des Petits Points) of the CNRS for the
financial support. The first author is member of the INdAM group GNSAGA.

\bigskip
\bibliographystyle{abbrv}
\def\cprime{$'$} \def\cprime{$'$} \def\cprime{$'$}

\end{document}